\documentclass[12pt,leqno]{article}
\usepackage[mathscr]{eucal}
\usepackage{amsmath,amssymb,latexsym,theorem,enumerate,ifthen}
\usepackage{orcidlink}

\newcommand{\NN}{\mathbb{N}}

\newcommand{\RR}{\mathbb{R}}

\newcommand{\ZZ}{\mathbb{Z}}

\newcommand{\be}{{\boldsymbol{e}}}

\newcommand{\bx}{{\boldsymbol{x}}}

\newcommand{\by}{{\boldsymbol{y}}}

\newcommand{\ee}{\mathrm{e}}
\newcommand{\ii}{\mathrm{i}}

\newcommand{\comment}[1]{}

\renewcommand{\mid}{\,|\,}

\renewcommand{\leq}{\leqslant}

\newcommand{\proofend}{\hfill\mbox{$\Box$}}

\numberwithin{equation}{section}

\theoremstyle{change} \theorembodyfont{\em}
\newtheorem{Lem}{Lemma.}[section]
\newtheorem{Thm}[Lem]{Theorem.}
\newtheorem{Pro}[Lem]{Proposition.}

\newtheorem{Cor}[Lem]{Corollary.}

\theorembodyfont{\rm}
\newtheorem{Def}[Lem]{Definition.}
\newtheorem{Rem}[Lem]{Remark.}

\def\OnlyOnArXiv#1#2{\ifthenelse{\equal{#1}{Y}}{#2}{}}

\long\def\Eq#1#2{\ifthenelse{\equal{#1}{*}}
  {\begin{equation*}\begin{aligned}#2\end{aligned}\end{equation*}}
  {\begin{equation}\begin{aligned}\label{#1}#2\end{aligned}\end{equation}}}

\newenvironment{proof}{\noindent{\bf Proof.}}{\proofend}

\begin{document}

\begin{center}
 {\bfseries\Large Jensen convex functions and doubly stochastic matrices}

\vspace*{3mm}

{\sc\large
  M\'aty\'as $\text{Barczy}^{*,\orcidlink{0000-0003-3119-7953}}$,
  Zsolt $\text{P\'ales}^{**,\diamond,\orcidlink{0000-0003-2382-6035}}$ }

\end{center}

\vskip0.2cm

\noindent
 * HUN-REN–SZTE Analysis and Applications Research Group,
   Bolyai Institute, University of Szeged,
   Aradi v\'ertan\'uk tere 1, H--6720 Szeged, Hungary.

\noindent
 ** Institute of Mathematics, University of Debrecen,
    Pf.~400, H--4002 Debrecen, Hungary.

\noindent E-mails: barczy@math.u-szeged.hu (M. Barczy),
                  pales@science.unideb.hu  (Zs. P\'ales).

\noindent $\diamond$ Corresponding author.

\vskip0.2cm


{
\renewcommand{\thefootnote}{}
\vspace*{0.2cm}
}

\vspace*{-10mm}

\begin{abstract}
Given an $n\times n$ doubly stochastic matrix $P$ satisfying an appropriate condition of linear algebraic-type, and a function $f$ defined on a nonempty interval, we show that the validity of a convexity-type functional
 inequality for $f$ in terms $P$ implies that $f$ is Jensen convex.
We also prove that if $f$ is convex, then the functional inequality in question holds for all doubly stochastic matrices of any order.
The particular case when the doubly stochastic matrix is a circulant one is also considered.
\end{abstract}


\section{Introduction}
\label{section_intro}

The theory of convex functions is an old and important field of classical analysis,
 and it has a large number of applications in many branches of mathematics. 
For a recent book containing numerous aspects, generalizations and applications connected to convex functions, 
 see, e.g., Niculescu and Persson \cite{NicLar}. 
In this paper, we establish some new connections between Jensen convex functions defined on a nonempty open interval  
 and doubly stochastic matrices in terms of a functional inequality of convexity type,
 and we also specialize our results to the setting of doubly stochastic circulant matrices.

For a comprehensive treatment of doubly stochastic matrices and Schur convex functions, we refer to the monograph \cite{MarOlk} by Marshall and Olkin.
 Concerning some recent inequalities for convex functions and doubly stochastic matrices,
 we can mention the paper \cite{Nie} of Niezgoda (see also the subsequent ones that can be found in the literature). 
According to our knowledge, in all of these papers,
 given a convex function and a doubly stochastic matrix, 
Niezgoda derived some inequalities for the given convex function, but we could not find any result for the reversed directions which would be similar to the assertion (i) of our main Theorem \ref{Thm_Jensen_convexity_with_DS_weights}.
 
Our paper is organized as follows.
In Section \ref{Sec_Jensen_convex_doubly_stoch}, as the main result of the paper, 
 we show that, given an $n\times n$ doubly stochastic matrix $P$ satisfying a certain condition of linear algebraic type, 
 and a function $f$ defined on a nonempty interval, if a convexity-type functional inequality holds for $f$ in terms of $P$, then $f$ has to be Jensen convex
 (see part (i) of our main Theorem \ref{Thm_Jensen_convexity_with_DS_weights}). 
Furthermore, we also prove that if $f$ is convex, then the functional inequality in question holds for all doubly stochastic matrices of any order 
 (see part (ii) of Theorem \ref{Thm_Jensen_convexity_with_DS_weights}). 
In Remark \ref{Rem_Schur}, we point out a connection between Theorem \ref{Thm_Jensen_convexity_with_DS_weights}
 and the notion of Schur convexity with respect to the doubly stochastic matrix with identical entries.
Section \ref{Sec_Jensen_convex_circulant} is devoted to derive some corollaries of Theorem \ref{Thm_Jensen_convexity_with_DS_weights} for doubly stochastic circulant matrices
 (see Proposition \ref{Pro_Jensen_convexity_with_circulant_weights} and Corollary \ref{Cor_quasi_arithmetic}). 
We also study in detail the cases corresponding to
 $2\times 2$, $3\times 3$ and $4\times 4$ doubly stochastic circulant matrices whose entries are not identically equal to each other, respectively
 (see Corollaries \ref{Cor_n=2}, \ref{Cor_n=3} and Proposition \ref{Pro_circulant_n=4}).
 
Throughout this paper, let $\NN$, $\ZZ_+$, $\RR$ and $\RR_+$ denote the sets of positive integers, non-negative integers, real numbers and  non-negative real numbers, respectively.
For a finite subset $A$ of $\NN$, its cardinality is denoted by $\vert A \vert$.
An interval $I\subseteq\RR$ will be called nondegenerate if it contains at least two distinct points.
Given a non-degenerate interval $I\subseteq\RR$, the difference $I-I$ stands for the set $\{x-y : x,y\in I\}$.
To avoid misunderstandings, in some case we write $(I-I)$ instead of $I-I$.
The natural basis in $\RR^n$ is denoted by $\{\be_1,\ldots,\be_n\}$, where $n\in\NN$.

Next, we recall the notion of $t$-convex functions (where $t\in[0,1]$).

\begin{Def}\label{Def_convexity}
Let $I\subseteq \RR$ be a nondegenerate interval and $t\in[0,1]$.
 A function $f:I\to\RR$ is called $t$-convex if
 \[
   f(tx+(1-t)y) \leq tf(x) + (1-t)f(y), \qquad x,y\in I.
 \]
If $f$ is $\frac{1}{2}$-convex, then it is called Jensen convex or midpoint convex as well. 
If $f$ is $t$-convex for all $t\in(0,1)$, then it is said to be convex.
Note that every function $f:I\to\RR$ is automatically $0$-convex and $1$-convex as well.
\end{Def}

In the next result, we collect some known implications among the convexity notions introduced in Definition \ref{Def_convexity}.

\begin{Thm}\label{Thm_Convex}
Let $I\subseteq \RR$ be a nonempty open interval and $f:I\to\RR$. Then the following two assertions hold.
\begin{enumerate}[(i)]
 \item If $f$ is $t$-convex for some $t\in(0,1)$, then it is Jensen convex.
 \item If $f$ is $t$-convex for some $t\in(0,1)$ and $f$ is bounded from above on some nonempty open subinterval of $I$, then $f$ is convex and continuous.
\end{enumerate}
\end{Thm}

For a proof of Theorem \ref{Thm_Convex}, see the proof of Theorem 1.1 in Barczy and P\'ales \cite{BarPal_2025}.

\section{Jensen convexity and doubly stochastic matrices}\label{Sec_Jensen_convex_doubly_stoch}

As the main result of the paper, we establish a connection between Jensen convex functions and doubly stochastic matrices
in terms of a functional inequality of convexity-type.

\begin{Thm}\label{Thm_Jensen_convexity_with_DS_weights}
Let $I\subseteq \RR$ be a nonempty open interval and $f:I\to\RR$ be a function.
\begin{itemize}
  \item[(i)] Let $n\in\NN$ and let $P=(p_{i,j})\in\RR^{n\times n}$ be a doubly stochastic matrix and assume that there exist disjoint nonempty subsets $A,B\subseteq\{1,\dots,n\}$ such that $\frac{1}{|A|}\sum_{\alpha\in A}\be_\alpha-\frac{1}{|B|}\sum_{\beta\in B}\be_\beta$ belongs to the linear span of the column vectors of $P$.
             If the inequality
  \begin{align}\label{help1_DS}
    f\left(\frac{x_1+\cdots+x_n}{n}\right)
    \leq \frac{1}{n}\sum_{i=1}^{n} f(p_{i,1}x_1+\cdots +p_{i,n}x_n)
  \end{align}
  holds for all $x_1,\ldots,x_n\in I$, then the function $f$ is Jensen convex on $I$.
  \item[(ii)] If the function $f$ is Jensen convex on $I$, then the inequality \eqref{help1_DS} holds for all $n\in\NN$, for all doubly stochastic matrices $P=(p_{i,j})\in\RR^{n\times n}$ and for all $x_1,\ldots,x_n\in I$.
\end{itemize}
\end{Thm}

\begin{proof}
First, we prove part (i). Let $n\in\NN$, let $P=(p_{i,j})\in\RR^{n\times n}$ be a doubly stochastic matrix, and let $A,B\subseteq\{1,\dots,n\}$ disjoint nonempty sets such that $\frac{1}{|A|}\sum_{\alpha\in A}\be_\alpha-\frac{1}{|B|}\sum_{\beta\in B}\be_\beta$ belongs to the linear span of the column vectors of $P$.
Therefore, there exists $\bx_{A,B}\in\RR^n$ such that $\frac{1}{|A|}\sum_{\alpha\in A}\be_\alpha-\frac{1}{|B|}\sum_{\beta\in B}\be_\beta=P{\bx_{A,B}}$. 
For the sake of simplicity, denote $a:=|A|$ and $b:=|B|$.
Suppose that \eqref{help1_DS} holds for all $x_1,\ldots,x_n\in I$.

Given $\bx=(x_1,\ldots,x_n)\in I^n$, introduce the vector  $\by=(y_1,\dots,y_n)$ by
\Eq{x-to-y}{
  y_i:=p_{i,1}x_1+\cdots +p_{i,n}x_n, \qquad i\in\{1,\dots,n\}.
}
In other words, $\by=P\bx$.
Using that the rows of $P$ are finite probability distributions, it follows that $y_i$ is a convex combination of $x_1,\dots,x_n$ for each $i\in\{1,\dots,n\}$, and hence $y_i\in I$, $i\in\{1,\dots,n\}$, in other words, $\by\in I^n$.

Using also that the column-sums of $P$ are equal to $1$, we have that
\Eq{*}{
  y_1+\cdots+y_n=x_1+\cdots+x_n,
}
and hence \eqref{help1_DS} can be rewritten as
\Eq{help1_y}{
  f\left( \frac{y_1+\cdots+y_n}{n} \right)
       \leq \frac{1}{n}\sum_{i=1}^{n}f(y_i)
}
for all $\by=(y_1,\dots,y_n)\in \{P\bx\mid\bx\in I^n\}\subseteq I^n$.

We are going to show that $f$ is locally Jensen convex on $I$, i.e., for all $p\in I$, there exists $r>0$ such that $(p-r,p+r)\subseteq I$ and $f$ is Jensen convex on $(p-r,p+r)$. Let $p\in I$ be fixed.

For $u,v\in I$, define
\Eq{help_y_def}{
   y_i:=\begin{cases}
        u & \text{if } i\in A,\\
        v & \text{if } i\in B,\\[1mm]
        \dfrac{au+bv}{a+b} & \text{if } i\in \{1,\dots,n\}\setminus(A\cup B).
        \end{cases}
}
Using that the map $\RR^n\ni\bx\mapsto P\bx$ is linear and the diagonal elements of $I^n$ are fixed points of this map, we can obtain that
\begin{align}\label{help13}
 \begin{split}
  (y_1,\dots,y_n)
  &=\frac{au+bv}{a+b}(\be_1+\dots+\be_n)
   +\frac{ab(u-v)}{a+b}\bigg(\frac{1}{a}\sum_{\alpha\in A}\be_\alpha-\frac{1}{b}\sum_{\beta\in B}\be_\beta\bigg)\\
  &=\frac{au+bv}{a+b} P( \be_1+\dots+\be_n )
          + \frac{ab(u-v)}{a+b} P\bx_{A,B}\\
  &=P\bigg(\frac{au+bv}{a+b}(\be_1+\dots+\be_n)
  +\frac{ab(u-v)}{a+b}\bx_{A,B}\bigg).
 \end{split}
 \end{align}
Further, we have that
\Eq{*}{
  \lim_{(u,v)\to(p,p)}
   \bigg(\frac{au+bv}{a+b}(\be_1+\dots+\be_n)
  +\frac{ab(u-v)}{a+b}\bx_{A,B}\bigg)
  =p(\be_1+\dots+\be_n)=(p,\dots,p)\in I^n.
}
Consequently, since $I$ and $I^n$ are open, there exists $r>0$ such that $(p-r,p+r)\subseteq I$ and
\Eq{*}{
   \frac{au+bv}{a+b}(\be_1+\dots+\be_n)
  +\frac{ab(u-v)}{a+b}\bx_{A,B}\in I^n
  \qquad \text{for all $u,v\in(p-r,p+r)$},
}
and hence, by \eqref{help13},
\Eq{*}{
  (y_1,\dots,y_n)\in \{P\bx\mid\bx\in I^n\} \qquad \text{for all $u,v\in(p-r,p+r)$,}
}
where $(y_1,\ldots,y_n)$ is defined in \eqref{help_y_def}.
We have that
\Eq{*}{
   \frac{y_1+\dots+y_n}{n}
   =\frac{1}{n}\bigg(au+bv+(n-a-b)\frac{au+bv}{a+b}\bigg)
   =\frac{au+bv}{a+b}.
}
Using this identity, the inequality \eqref{help1_y}
takes the form
 \[
 f\left( \frac{au+bv}{a+b} \right)
       \leq \frac{1}{n}\bigg(af(u)+bf(v)
       +(n-a-b)f\left(\frac{au+bv}{a+b} \right)\bigg),
 \]
which reduces to
 \Eq{2L+}{
  f\left(\frac{au+bv}{a+b}\right)
       \leq \frac{af(u)+bf(v)}{a+b}
 }
for all $u,v\in (p-r,p+r)$, that is, $f$ is $\frac{a}{a+b}$-convex on $(p-r,p+r)$, where $\frac{a}{a+b}\in(0,1)$.
According to part (i) of Theorem \ref{Thm_Convex}, it follows that $f$ is Jensen convex on $(p-r,p+r)$.

All in all, we have proved that, for all $p\in I$, there exists $r>0$ such that $(p-r,p+r)\subseteq I$ and
$f$ is Jensen convex on $(p-r,p+r)$,
in other words, $f$ is locally Jensen convex on $I$.
This implies that $f$ is Jensen convex on $I$, i.e.,
\[
   f\left(\frac{u+v}{2}\right)
       \leq \frac{f(u)+f(v)}{2},\qquad u,v\in I,
\]
see, for example, Corollary 2 with the choices $n=2$ and $t_1=t_2=1$ in Gil\'anyi and P\'ales \cite{GilPal01} or Corollary 6 with the choice $t=\frac{1}{2}$ in Nikodem and P\'ales \cite{NikPal03}
(which, roughly speaking, states that Jensen convexity is a localizable convexity property).

Finally, we prove part (ii).
Suppose that the function $f:I\to\RR$ is Jensen convex.
Then, for all $n\in\NN$, for all doubly stochastic matrices $P=(p_{i,j})\in\RR^{n\times n}$ and for all $x_1,\ldots,x_n\in I$,
 we have that
  \begin{align*}
          \sum_{i=1}^n  (p_{i,1}x_1+\cdots +p_{i,n}x_n)
           =  x_1+\cdots + x_n,
         \end{align*}
 yielding that
 \begin{align*}
   \frac{1}{n} \sum_{k=1}^n x_k
                            =   \frac{1}{n} \sum_{i=1}^n  (p_{i,1}x_1+\cdots +p_{i,n}x_n).
 \end{align*}
Now the Jensen convexity of $f$ together with Lemma 5.3.1 in Kuczma \cite{Kuc2009}
 imply that the inequality \eqref{help1_DS} holds.
\end{proof}

In the next remark, among others, we provide a sufficient condition under which the condition on $P$ in part (i) of Theorem \ref{Thm_Jensen_convexity_with_DS_weights} holds, and we also formulate an open problem.

\begin{Rem}\label{Rem_Jensen_convexity_with_DS_weights}
(i). If $P\in\RR^{n\times n}$ is an invertible doubly stochastic matrix, then, for all disjoint nonempty subsets $A,B\subseteq\{1,\dots,n\}$, we have that $\frac{1}{|A|}\sum_{\alpha\in A}\be_\alpha-\frac{1}{|B|}\sum_{\beta\in B}\be_\beta$ belongs to the linear span of the column vectors of $P$, i.e., the condition in part (i) of Theorem \ref{Thm_Jensen_convexity_with_DS_weights} holds trivially.
Note also that, for any disjoint nonempty subsets $A,B\subseteq\{1,\ldots,n\}$, 
 the vectors $\frac{1}{|A|}\sum_{\alpha\in A}\be_\alpha-\frac{1}{|B|}\sum_{\beta\in B}\be_\beta$ and $\be_1+\cdots+\be_n$ are linearly independent, and $\be_1+\cdots+\be_n$ belongs to the linear span of the column vectors of $P$ (since the row-sums of $P$ are equal to $1$).
Consequently, under the condition on $P$ in part (i) of Theorem \ref{Thm_Jensen_convexity_with_DS_weights},
 the rank of $P$ is at least $2$.
However, we do not know whether the assertion in part (i) of Theorem \ref{Thm_Jensen_convexity_with_DS_weights} 
 holds for every doubly stochastic matrix having rank at least 2.
We leave this question as an open problem.

(ii).
If $n\in\NN$, then, by the trinomial theorem, the number of ways how one can choose disjoint (possibly empty) subsets $A$ and $B$ of $\{1,\ldots,n\}$ is
 \[
 \sum_{a,b\in\ZZ_+\,:\,a+b\leq n} \binom{n}{a}\binom{n-a}{b}
      =  \sum_{a,b\in\ZZ_+\,:\,a+b\leq n} \frac{n!}{a!b!(n-a-b)!} 1^a 1^b 1^{n-a-b}
      =(1+1+1)^3=3^n.
 \]
On the other hand, the number of ways how one can choose a (possibly empty) subset ($A$ or $B$) of $\{1,\ldots,n\}$ is $\sum_{a=0}^n \binom{n}{a} = 2^n$.
Therefore, by the inclusion–exclusion principle (which is sometimes referred to as the sieve formula), the cardinality of the set of vectors
$\frac{1}{|A|}\sum_{\alpha\in A}\be_\alpha-\frac{1}{|B|}\sum_{\beta\in B}\be_\beta$ that satisfy the assumption in part (i) of Theorem \ref{Thm_Jensen_convexity_with_DS_weights}
 is $3^n-2\cdot2^n+1$, which is of order $3^n$.

(iii) If the assumptions of part (i) of Theorem \ref{Thm_Jensen_convexity_with_DS_weights} hold, $f$ is bounded from above on some nonempty open subset of $I$ and the inequality \eqref{help1_DS} is satisfied for all $x_1,\ldots,x_n\in I$, then $f$ is convex on $I$.
This statement easily follows from part (i) of Theorem \ref{Thm_Jensen_convexity_with_DS_weights} combined with part (ii) of Theorem \ref{Thm_Convex}.
\proofend
\end{Rem}

The next remark is devoted to give a reformulation of the inequality \eqref{help1_DS} in Theorem \ref{Thm_Jensen_convexity_with_DS_weights},
 and to point out a connection between Theorem \ref{Thm_Jensen_convexity_with_DS_weights} and 
 the notion of Schur convexity with respect to the doubly stochastic matrix with identical entries.

\begin{Rem}\label{Rem_Schur}
Let $I\subseteq \RR$ be a nonempty open interval, $n\in\NN$, 
 and $f:I\to\RR$ be a function.
Let us define $F:I^n\to\RR$, 
\Eq{F_def}{
   F(x_1,\dots,x_n):=f(x_1)+\dots+f(x_n),\qquad (x_1,\ldots,x_n)\in I^n.
 }
Let $S_n$ denote the doubly stochastic matrix with identical entries equal to $\frac{1}{n}$.  
Then, for a doubly stochastic matrix $P\in\RR^{n\times n}$ and a vector $\bx\in I^n$, the inequality \eqref{help1_DS} in Theorem \ref{Thm_Jensen_convexity_with_DS_weights} 
 can be also written in the form
 \Eq{I}{
    F(S_n \bx)\leq F(P\bx).
 }

We claim that, for an arbitrary function $G:I^n\to \RR$, the inequality
 \Eq{I_b}{
    G(S_n \bx)\leq G(P\bx)
  }
is equivalent to 
 \Eq{III}{
 G(S_n\by)\leq G(\by), \qquad \by\in P(I^n)=\{P\bx\mid \bx\in I^n\}\subseteq I^n.
 }
First, suppose that \eqref{III} holds.
Then, using that $S_nP=S_n$ (since $P$ is doubly stochastic), with $\by:=P\bx$, where $\bx\in I^n$, we get that
 \[
 G(S_n\bx)=G(S_nP\bx)\leq G(P\bx),
 \]
 as desired.
Now suppose that \eqref{I_b} holds. 
Then, using again that $S_nP=S_n$, we have that 
 \[
   G(S_nP\bx)=G(S_n\bx)\leq G(P\bx), \qquad \bx\in I^n,
 \]
 and hence \eqref{III} is valid for all $\by\in P(I^n)$, as desired.
 
Furthermore, if the condition on $P$ in part (i) of Theorem \ref{Thm_Jensen_convexity_with_DS_weights} is satisfied
 and the inequality \eqref{I} holds for all $\bx\in I^n$ (where $F$ is defined by \eqref{F_def}), then 
  $f$ is Jensen convex on $I$ (following from Theorem \ref{Thm_Jensen_convexity_with_DS_weights}), 
  and, hence, using also Lemma 5.3.1 in Kuczma \cite{Kuc2009}, it implies that
 \begin{align}\label{II}
  \begin{split}
        F(S_n\bx)
          & = F\left(\frac{x_1+\cdots +x_n}{n},\ldots,\frac{x_1+\cdots + x_n}{n}\right)
           = nf\left(\frac{x_1+\cdots +x_n}{n}\right) \\
          & \leq n\cdot \frac{1}{n} \sum_{i=1}^n f(x_i)
           = \sum_{i=1}^n f(x_i) 
           = F(\bx), \qquad \bx\in I^n.
 \end{split}          
 \end{align}
According to Burai et al.\ \cite[Definition 3]{BurMakSzo},  if the inequality \eqref{II} holds, then we say that $F$ is $S_n$-Schur convex.
\proofend 
\end{Rem}

\section{Jensen convexity and doubly stochastic circulant matrices}\label{Sec_Jensen_convex_circulant}

In this section, we derive some corollaries of Theorem \ref{Thm_Jensen_convexity_with_DS_weights} for doubly stochastic
 circulant-type matrices (also called left-circulant matrices).
To simplify the formulation of our subsequent results, for each $n\in\NN$,  $i\in\{1,\dots,n\}$ and $j\in\{0,\dots,n-1\}$, let us define the truncated sum $i\oplus j$ of $i$ and $j$ by
       \begin{align}\label{def_truncated_sum}
         i\oplus j
         :=\begin{cases}
          i+j &\mbox{if } i+j\leq n,\\
          i+j-n &\mbox{if } i+j>n.
          \end{cases}
       \end{align}

\begin{Pro}\label{Pro_Jensen_convexity_with_circulant_weights}
Let $I\subseteq \RR$ be a nonempty open interval and $f:I\to\RR$ be a function.
\begin{itemize}
  \item[(i)] Let $n\in\NN$ and denote $\omega_n:=\ee^{\frac{2\pi}{n}\ii}$. Let $\lambda_1,\dots,\lambda_n\in[0,1]$ such that $\lambda_1+\dots+\lambda_n=1$ and
  \Eq{circ_cond}{
    \sum_{j=1}^n\lambda_j\omega_n^{k(j-1)}\neq0, \qquad k\in\{1,\dots,n-1\}.
  }
  If the inequality
             \begin{align}\label{help1_new}
                f\left(\frac{x_1+\cdots+x_n}{n}\right)
                \leq \frac{1}{n}\sum_{j=0}^{n-1} f( \lambda_{1\oplus j}x_1+\cdots+\lambda_{n\oplus j}x_n)
              \end{align}
  holds for all $x_1,\ldots,x_n\in I$, then the function $f$ is Jensen convex on $I$, where $\oplus$ is defined in \eqref{def_truncated_sum}.
  
  \item[(ii)] If the function $f$ is Jensen convex on $I$, then the inequality \eqref{help1_new} holds for all $n\in\NN$, for all $\lambda_1,\dots,\lambda_n\in[0,1]$ with $\lambda_1+\dots+\lambda_n=1$ and $x_1,\ldots,x_n\in I$.
\end{itemize}
\end{Pro}

\begin{proof}
Let $n\in\NN$ and $\lambda_1,\dots,\lambda_n\in[0,1]$ be such that $\lambda_1+\dots+\lambda_n=1$.
Let us define
 \begin{align}\label{help_circulant_P}
  P:= (\lambda_{i\oplus j})_{i=1,\ldots,n}^{j=0,\ldots,n-1}
    = \begin{bmatrix}
        \lambda_1 &\lambda_2 & \lambda_3 & \cdots &\lambda_{n-1}  & \lambda_n \\
        \lambda_2 & \lambda_3 & \lambda_4 & \cdots &  \lambda_n & \lambda_1 \\
        \vdots & \vdots & \vdots & \vdots & \vdots & \vdots \\
        \lambda_n & \lambda_1 & \lambda_2 & \cdots & \lambda_{n-2}& \lambda_{n-1} \\
      \end{bmatrix}
      \in\RR^{n\times n}.
 \end{align}
Then $P$ is a left circulant matrix with parameters $\lambda_1,\ldots,\lambda_n$, see Carmona et al. \cite[formula (5)]{CarEncGagJim}.
By part (vii) of Lemma 1.1 and Lemma 1.2 in Carmona et al.\ \cite{CarEncGagJim},
 we get that $P$ is invertible if and only if $\sum_{j=1}^n \lambda_j\omega_n^{k(j-1)}\ne 0$, $k\in\{0,\ldots,n-1\}$.
Since $\lambda_1+\cdots+\lambda_n=1$, it yields that $P$ is invertible if and only if \eqref{circ_cond} holds.
Further, $P$ is doubly stochastic as well due to that $\lambda_1+\cdots+\lambda_n=1$. Moreover, we have that 
 \begin{align*}
   \sum_{i=1}^n f\big( \lambda_{i\oplus 0} x_1 + \cdots + \lambda_{i\oplus (n-1)} x_n \big)
     & = \sum_{j=0}^{n-1} f\big( \lambda_{(j+1)\oplus 0} x_1 + \cdots + \lambda_{(j+1)\oplus (n-1)} x_n \big)\\
     & = \sum_{j=0}^{n-1} f\big( \lambda_{1\oplus j} x_1 + \cdots + \lambda_{n\oplus j} x_n \big).
 \end{align*}
Therefore the statement follows from Theorem \ref{Thm_Jensen_convexity_with_DS_weights}
 taking into account part (i) of Remark \ref{Rem_Jensen_convexity_with_DS_weights} as well.
\end{proof}

In the first part of the next remark, we point out the fact that
 if, in addition, $f$ is bounded from above on some nonempty open subset of $I$,
 then in part (i) of Proposition \ref{Pro_Jensen_convexity_with_circulant_weights}
 we can obtain that $f$ is convex on $I$.
The second part of next remark is about the case when the condition \eqref{circ_cond} does not hold.

\begin{Rem}\label{Rem_circular_doubly_stochastic}
(i). If the assumptions of part (i) of Proposition \ref{Pro_Jensen_convexity_with_circulant_weights} hold,
 $f$ is bounded from above on some nonempty open subset of $I$ and the inequality \eqref{help1_new} is satisfied
 for all $x_1,\ldots,x_n\in I$, then $f$ is convex on $I$.
This statement easily follows from part (i) of Proposition \ref{Pro_Jensen_convexity_with_circulant_weights} combined with part (ii) of Theorem \ref{Thm_Convex}.

(ii).
Let $I\subseteq \RR$ be a nonempty open interval, $f:I\to\RR$ be a function,
 $n\in\NN\setminus\{1\}$, and $\lambda_1,\dots,\lambda_n\in[0,1]$ be such that $\lambda_1+\dots+\lambda_n=1$.
Let us consider the matrix $P$ given by \eqref{help_circulant_P}.
By the proof of Proposition \ref{Pro_Jensen_convexity_with_circulant_weights},
 $P$ is not invertible if and only if
 the condition \eqref{circ_cond} does not hold, i.e., the set
 \Eq{*}{
    K:=\bigg\{k\in\{1,\dots,n-1\}\colon\sum_{j=1}^n\lambda_j\omega_n^{k(j-1)}=0\bigg\}
 }
 is nonempty.
Since
\Eq{*}{
  \omega_n^{k(j-1)}=\exp\left(\frac{2k(j-1)\pi}{n}\ii\right)
  =\cos\left(\frac{2k(j-1)\pi}{n}\right)+
  \ii\sin\left(\frac{2k(j-1)\pi}{n}\right),\quad j\in\{1,\ldots,n\},
}
 we get that
\Eq{help_circulant_spec_1}{
    \sum_{j=1}^n\lambda_j\cos\left(\frac{2k(j-1)\pi}{n}\right)=0,\qquad \sum_{j=1}^n\lambda_j\sin\left(\frac{2k(j-1)\pi}{n}\right)=0,
    \qquad k\in K.
}
Observe that, if $n$ is even and $k=n/2$, then the second condition in \eqref{help_circulant_spec_1} is trivial, 
 while the first condition in \eqref{help_circulant_spec_1} together with $\lambda_1+\dots+\lambda_n=1$ yields
 \begin{align}\label{help_circ_doubly_stoch_1}
  \lambda_1+\lambda_3+\dots+\lambda_{n-1}
  =\lambda_2+\lambda_4+\dots+\lambda_{n}=\frac12.
 \end{align}
\proofend
\end{Rem}

Next, we present a corollary of Proposition \ref{Pro_Jensen_convexity_with_circulant_weights}.  

\begin{Cor}\label{Cor_quasi_arithmetic}
Let $I\subseteq \RR$ be a nonempty open interval and $f:I\to\RR$ be a function.
\begin{itemize}
  \item[(i)] Let $n\in\NN\setminus\{1\}$.
             If the inequality
             \begin{align}\label{help1}
                \begin{split}
              f\bigg(\frac{x_1+\cdots+x_n}{n} \bigg)
                  \leq \frac{1}{n}\sum_{i=1}^n f\bigg(\frac{x_1+\cdots+x_n - x_i}{n-1} \bigg) 
                \end{split}
              \end{align}
             holds for all $x_1,\ldots,x_n\in I$, then the function $f$ is Jensen convex on $I$.
  \item[(ii)] If the function $f$ is Jensen convex on $I$, then the inequality \eqref{help1} holds for all $n\in\NN\setminus\{1\}$ and $x_1,\ldots,x_n\in I$.
 \end{itemize}
\end{Cor}

\begin{proof}
First, we prove part (i). Let $n\in\NN\setminus\{1\}$.
Let us apply part (i) of Proposition \ref{Pro_Jensen_convexity_with_circulant_weights} with choices 
  $\lambda_1=\cdots=\lambda_{n-1}:=\frac{1}{n-1}$ and $\lambda_n:=0$.
Then, for each $k\in\{1,\ldots,n-1\}$, we have that
 \begin{align*}
   \sum_{j=1}^n\lambda_j\omega_n^{k(j-1)}
    = \frac{1}{n-1} \sum_{j=1}^{n-1} \ee^{\frac{2\pi}{n}k(j-1)\ii}
    = \frac{1}{n-1} \frac{\ee^{\frac{2\pi}{n}k(n-1)\ii} -1}{\ee^{\frac{2\pi}{n}k\ii} - 1 }
    \ne 0,
 \end{align*}
 yielding that condition \eqref{circ_cond} holds.
Further, for all $x_1,\ldots,x_n\in I$ and each $j\in\{0,\ldots,n-1\}$, we have that
 \[
   \lambda_{1\oplus j}x_1+\cdots+\lambda_{n\oplus j}x_n = \frac{x_1+\cdots+x_n-x_{j_*}}{n-1},
 \]
 where $j_*\in\{1,\ldots,n\}$ is such that $j_*\oplus j=n$.
Taking into account that $\{x_{j^*} : j\in\{0,\ldots,n-1\}\} = \{x_1,\ldots,x_n\}$,
 part (i) of Proposition \ref{Pro_Jensen_convexity_with_circulant_weights} implies part (i).

The second assertion follows immediately from assertion (ii) of Proposition \ref{Pro_Jensen_convexity_with_circulant_weights}with choices $\lambda_1=\cdots=\lambda_{n-1}:=\frac{1}{n-1}$ and $\lambda_n:=0$.
\end{proof}

\begin{Rem}
If the assumptions of Corollary \ref{Cor_quasi_arithmetic} hold,
 $f$ is bounded from above on some nonempty open subset of $I$, $n\in\NN\setminus\{1\}$, 
 and the inequality \eqref{help1} is satisfied
 for all $x_1,\ldots,x_n\in I$, then $f$ is convex on $I$.
 Indeed, by part (i) of Corollary \ref{Cor_quasi_arithmetic}, $f$ is Jensen convex on $I$,
 and, using part (ii) of Theorem \ref{Thm_Convex}, it implies that $f$ is convex on $I$, as desired.
\proofend
\end{Rem}

Next, we specialize Proposition \ref{Pro_Jensen_convexity_with_circulant_weights} to the cases $n=2$ and $n=3$, respectively.

\begin{Cor}\label{Cor_n=2}
Let $I\subseteq \RR$ be a nonempty open interval and $f:I\to\RR$ be a function.
Let $\lambda_1,\lambda_2\in[0,1]$ be such that $\lambda_1+\lambda_2=1$.
\begin{itemize}
  \item[(i)] The condition \eqref{circ_cond} in case of $n=2$ is equivalent to
             $\lambda_1\ne\lambda_2$, i.e., not both of $\lambda_1$ and $\lambda_2$ are equal to $\frac{1}{2}$.             
   \item[(ii)] If $\lambda_1\ne \lambda_2$ and the inequality \eqref{help1_new} with $n=2$
              holds for all $x_1,x_2\in I$, i.e.,
              \[
               f\left(\frac{x_1+x_2}{2}\right)
                  \leq \frac{1}{2}\Big( f(\lambda_1x_1+\lambda_2x_2) + f(\lambda_2x_1+\lambda_1x_2) \Big),
                  \qquad x_1,x_2\in I, 
              \]
              then the function $f$ is Jensen convex on $I$.
\end{itemize}
\end{Cor}

\begin{proof}
(i). Since $\omega_2=\ee^{\pi \ii} = -1$, the condition \eqref{circ_cond} in case of $n=2$ is equivalent to 
 $\lambda_1-\lambda_2\ne 0$, i.e., $\lambda_1\ne \lambda_2$.
Since $\lambda_1+\lambda_2=1$, we have that $\lambda_1\ne \lambda_2$ is equivalent to the fact that  
 not both of $\lambda_1$ and $\lambda_2$ are equal to $\frac{1}{2}$.

(ii). Assume that $\lambda_1\ne \lambda_2$ .
Then, by part (i), the condition \eqref{circ_cond} in case of $n=2$ holds, 
 and hence part (i) of Proposition \ref{Pro_Jensen_convexity_with_circulant_weights} yields the assertion.
\end{proof}

\begin{Cor}\label{Cor_n=3}
Let $I\subseteq \RR$ be a nonempty open interval and $f:I\to\RR$ be a function.
Let $\lambda_1,\lambda_2,\lambda_3\in[0,1]$ be such that $\lambda_1+\lambda_2+\lambda_3=1$.
\begin{itemize}
  \item[(i)] The condition \eqref{circ_cond} in case of $n=3$ is equivalent that
             $\lambda_1,\lambda_2,\lambda_3$ are not all equal to each other (i.e., 
             not all of them are equal to $\frac{1}{3}$).             
   \item[(ii)] If $\lambda_1,\lambda_2,\lambda_3$ are not all equal to each other and the inequality \eqref{help1_new} with $n=3$
              holds for all $x_1,x_2,x_3\in I$, then the function $f$ is Jensen convex on $I$.
\end{itemize}
\end{Cor}

\begin{proof}
(i).
For each $k\in\{1,2\}$, we have that
 \[
   \omega_3^{k(j-1)}
      =\exp\left(\frac{2k(j-1)\pi}{3}\ii\right)
      = \begin{cases}
            1 & \text{if $j=1$,}\\
            \ee^{\frac{2k\pi}{3}\ii} = \left( -\frac{1}{2} + \frac{\sqrt{3}}{2}\ii\right)^k & \text{if $j=2$,}\\
            \ee^{\frac{4k\pi}{3}\ii} = \left( -\frac{1}{2} - \frac{\sqrt{3}}{2}\ii \right)^k & \text{if $j=3$.}
        \end{cases}
 \] 
Consequently, the condition \eqref{circ_cond} in case of $n=3$ is  equivalent to
 \[
  \lambda_1+ \left( -\frac{1}{2} + \frac{\sqrt{3}}{2}\ii\right)^k \lambda_2 + \left( -\frac{1}{2} - \frac{\sqrt{3}}{2}\ii \right)^k \lambda_3 \ne 0, \qquad
    k\in\{1,2\}, 
 \]
 i.e.,
 \begin{align*}
   2\lambda_1 + (-1+\sqrt{3}\ii)\lambda_2 - (1+\sqrt{3}\ii)\lambda_3 \ne 0
    \qquad \text{and}\qquad 
   2\lambda_1 - (1+\sqrt{3}\ii)\lambda_2 + (-1+\sqrt{3}\ii)\lambda_3 \ne 0. 
 \end{align*}
Using that $\lambda_1=1-\lambda_2-\lambda_3$, these conditions are equivalent to 
 \begin{align*}
   (2-3\lambda_2-3\lambda_3) + \sqrt{3}(\lambda_2-\lambda_3) \ii \ne 0
    \qquad \text{and}\qquad 
   (2-3\lambda_2-3\lambda_3) + \sqrt{3}(-\lambda_2+\lambda_3) \ii \ne 0. 
 \end{align*}
Here the first condition (and the second one as well) is equivalent to 
 \[
   (2-3\lambda_2-3\lambda_3)^2 + (\lambda_2-\lambda_3)^2>0.
 \]
Therefore, the condition \eqref{circ_cond} in case of $n=3$ does not hold if and only if
 $\lambda_2+\lambda_3=\frac{2}{3}$ and $\lambda_2=\lambda_3$, i.e., 
 if and only if $\lambda_2=\lambda_3=\frac{1}{3}$.
Hence, since $\lambda_1=1-\lambda_2-\lambda_3$, we have that  
 the condition \eqref{circ_cond} in case of $n=3$ does not hold if and only if
 $\lambda_1=\lambda_2=\lambda_3=\frac{1}{3}$,
 which yields part (i).
 
We mention that one could give an alternative proof as well.
By part (vii) of Lemma 1.1 and Lemma 1.2 in Carmona et al.\ \cite{CarEncGagJim}, we get that $P$ (introduced in \eqref{help_circulant_P}) is invertible if and only if 
 condition \eqref{circ_cond} holds.
Then one could calculate directly the determinant of $P$ in case of $n=3$, and study when it can be zero.
The determinant in question can be found, e.g., in Carmona et al.\ \cite[page 787]{CarEncGagJim}. 

(ii). Assume that $\lambda_1,\lambda_2,\lambda_3$ are not equal to each other.
Then, by part (i), the condition \eqref{circ_cond} in case of $n=3$ holds,  and hence part (i) of Proposition \ref{Pro_Jensen_convexity_with_circulant_weights} yields the assertion.
\end{proof}

In the next proposition, on the one hand,
 we make the condition \eqref{circ_cond} in case of $n=4$ more explicit, and, on the other hand,
 we prove part (i) of Proposition \ref{Pro_Jensen_convexity_with_circulant_weights} 
 without this condition for a $4\times 4$ doubly stochastic matrix whose entries are not identically equal to $\frac{1}{4}$.

\begin{Pro}\label{Pro_circulant_n=4}
Let $I\subseteq \RR$ be a nonempty open interval and $f:I\to\RR$ be a function.
Let $\lambda_1,\dots,\lambda_4\in[0,1]$ be such that $\lambda_1+\dots+\lambda_4=1$.
\begin{itemize}
  \item[(i)] The condition \eqref{circ_cond} in case of $n=4$ is equivalent to
             \begin{align}\label{help12}
              (\lambda_1-\lambda_3)^2+(\lambda_2-\lambda_4)^2>0 \qquad \text{and}\qquad \lambda_1+\lambda_3\neq\lambda_2+\lambda_4.
             \end{align}
   \item[(ii)] If $\lambda_i$, $i\in\{1,\ldots,4\}$, are not all equal to each other (i.e,, not all of them are equal to $\frac{1}{4}$)
               and the inequality \eqref{help1_new} with $n=4$
              holds for all $x_1,\ldots,x_4\in I$, then the function $f$ is Jensen convex on $I$.
\end{itemize}
\end{Pro}

Note that, in general, the two conditions in \eqref{help12} are not dependent on each other, since, for example, in the case of $\lambda_1=\frac{1}{4}$, $\lambda_2=\frac{1}{6}$,
 $\lambda_3=\frac{1}{4}$ and $\lambda_4=\frac{1}{3}$, we have $(\lambda_1-\lambda_3)^2+(\lambda_2-\lambda_4)^2=\frac{1}{36}$,
 but $\lambda_1+\lambda_3=\frac{1}{2}=\lambda_2+\lambda_4$.
On the other hand, if $\lambda_1=\frac{1}{3}$, $\lambda_2=\frac{1}{6}$, $\lambda_3=\frac{1}{3}$ and $\lambda_4=\frac{1}{6}$, then
 $(\lambda_1-\lambda_3)^2+(\lambda_2-\lambda_4)^2=0$, but $\lambda_1+\lambda_3=\frac{2}{3}\ne\frac{1}{3}=\lambda_2+\lambda_4$.

\noindent{\bf Proof of Proposition \ref{Pro_circulant_n=4}.}
(i).
For each $k\in\{1,2,3\}$, we have that
 \[
   \omega_4^{k(j-1)}
      =\exp\left(\frac{2k(j-1)\pi}{4}\ii\right)
      = \begin{cases}
            1 & \text{if $j=1$,}\\
            \ee^{\frac{k\pi}{2}\ii} = \ii^k & \text{if $j=2$,}\\
            \ee^{k\pi\ii} = (-1)^k & \text{if $j=3$,}\\
            \ee^{\frac{3k\pi}{2}\ii} = (-\ii)^k & \text{if $j=4$.}
        \end{cases}
 \]
Consequently, the condition \eqref{circ_cond} in case of $n=4$ is equivalent to
 \begin{align*}
   \lambda_1+\ii^k\lambda_2+(-1)^k\lambda_3+(-\ii)^k\lambda_4\ne 0,
     \qquad k\in\{1,2,3\},
 \end{align*}
 i.e.,
 $$
    \lambda_1+\ii\lambda_2-\lambda_3-\ii\lambda_4\neq0,\qquad
    \lambda_1-\lambda_2+\lambda_3-\lambda_4\neq0,\qquad
    \lambda_1-\ii\lambda_2-\lambda_3+\ii\lambda_4\neq0,
 $$
 where $\lambda_1,\dots,\lambda_4\in[0,1]$ are such that $\lambda_1+\dots+\lambda_4=1$.
Here the first condition (and the third condition as well) is equivalent to
 $(\lambda_1-\lambda_3)^2+(\lambda_2-\lambda_4)^2>0$,
 and the second condition is equivalent to $\lambda_1+\lambda_3\neq\lambda_2+\lambda_4$.
All in all, the condition \eqref{circ_cond} in case of $n=4$ is equivalent to \eqref{help12}.

(ii). If the condition \eqref{circ_cond} with $n=4$ holds, then the statement follows by part (i)
 of Proposition \ref{Pro_Jensen_convexity_with_circulant_weights}.
 
If the condition \eqref{circ_cond} with $n=4$ does not hold, then we are going to show that part (i) of Theorem \ref{Thm_Jensen_convexity_with_DS_weights} cab be applied.
Assume that the condition \eqref{circ_cond} in case of $n=4$ does not hold, i.e.,
 \[
   \sum_{j=1}^4\lambda_j\omega_4^{k(j-1)} = 0  \qquad \text{for some \ $k\in\{1,2,3\}$.}
 \]
Then, by part (i), we have that
 \[
    (\lambda_1-\lambda_3)^2+(\lambda_2-\lambda_4)^2 = 0 \qquad \text{or}\qquad \lambda_1+\lambda_3=\lambda_2+\lambda_4.
 \]

Case 1. If $(\lambda_1-\lambda_3)^2+(\lambda_2-\lambda_4)^2 = 0$, then $\lambda_1 = \lambda_3$ and $\lambda_2 = \lambda_4$,
 yielding that the matrix $P$ introduced in \eqref{help_circulant_P} takes the form
 \[
  P=\begin{bmatrix}
     \lambda_1 & \lambda_2 & \lambda_1 & \lambda_2 \\
     \lambda_2 & \lambda_1 & \lambda_2 & \lambda_1 \\
     \lambda_1 & \lambda_2 & \lambda_1 & \lambda_2 \\
     \lambda_2 & \lambda_1 & \lambda_2 & \lambda_1 \\
   \end{bmatrix},
 \]
 where $\lambda_1,\lambda_2\in\RR_+$ are such that $\lambda_1+\lambda_2=\frac{1}{2}$.
If $\lambda_1=\lambda_2$, i.e., $\lambda_1=\lambda_2 = \frac{1}{4}$ were true,
 then all the entries of $P$ would be equal to $\frac{1}{4}$, but we excluded this case.
Consequently, we have that $\lambda_1\ne \lambda_2$.
Then the linear span of the column vectors of $P$ is
 \[
  \left\{
    \alpha \begin{pmatrix}
             \lambda_1 \\
             \lambda_2 \\
             \lambda_1 \\
             \lambda_2 \\
           \end{pmatrix}
     + \beta \begin{pmatrix}
             \lambda_2 \\
             \lambda_1 \\
             \lambda_2 \\
             \lambda_1 \\
           \end{pmatrix}
       : \alpha,\beta\in\RR
  \right\}
  = \left\{
    \begin{pmatrix}
             u \\
             v \\
             u \\
             v \\
           \end{pmatrix}
       : u,v\in\RR
  \right\}
  =\Big\{ u(\be_1+\be_3) + v(\be_2+\be_4): u,v\in\RR\Big\},
 \]
 where, at the first equality, 
 we used that, for any $u,v\in\RR$, the system of linear equations $\alpha \lambda_1+\beta\lambda_2=u$ and
 $\alpha \lambda_2+\beta\lambda_1=v$ can be solved uniquely for $\alpha,\beta\in\RR$ (since $\lambda_1\ne\lambda_2$).
Consequently, the condition on $P$ in part (i) of Theorem \ref{Thm_Jensen_convexity_with_DS_weights}
 is satisfied by choosing $A:=\{1,3\}$ and $B:=\{2,4\}$, and hence
 we can apply part (i) of Theorem \ref{Thm_Jensen_convexity_with_DS_weights} in case of $\lambda_1 = \lambda_3$ and $\lambda_2 = \lambda_4$
 with $\lambda_1\ne\lambda_2$.

Case 2. If $\lambda_1+\lambda_3=\lambda_2+\lambda_4$, then $\lambda_4 = \lambda_1+\lambda_3 - \lambda_2$, and hence
 the matrix $P$ introduced in \eqref{help_circulant_P} takes the form
 \[
  P=\begin{bmatrix}
     \lambda_1 & \lambda_2 & \lambda_3 & \lambda_1+\lambda_3 - \lambda_2 \\
     \lambda_2 & \lambda_3 & \lambda_1+\lambda_3 - \lambda_2 & \lambda_1 \\
     \lambda_3 & \lambda_1+\lambda_3 - \lambda_2 & \lambda_1 & \lambda_2 \\
     \lambda_1+\lambda_3 - \lambda_2 & \lambda_1 & \lambda_2 & \lambda_3 \\
   \end{bmatrix}.
 \]
We will verify that $\be_1+\be_2-\be_3-\be_4$ is in the linear span of the column vectors of $P$.
Since the fourth row (column) of $P$ is nothing else but the sum of its first and third rows (columns) 
 minus its second row (column) (and hence the rank of $P$ is at most $3$), it is enough to check that the system of
 linear equations
 \begin{align}\label{help_lin_sys}
   \begin{bmatrix}
     \lambda_1 & \lambda_2 & \lambda_3 \\
     \lambda_2  & \lambda_3  & \lambda_1+\lambda_3 - \lambda_2 \\
     \lambda_3  & \lambda_1+\lambda_3 - \lambda_2 & \lambda_1 \\
   \end{bmatrix}
   \begin{bmatrix}
     \alpha \\
     \beta \\
     \gamma \\
   \end{bmatrix}
  =  \begin{bmatrix}
     1 \\
     1 \\
     -1 \\
   \end{bmatrix}
 \end{align}
 can be solved for $\alpha,\beta,\gamma\in\RR$.
One can calculate that the determinant of the $3\times 3$ (base) matrix of the system of linear equations \eqref{help_lin_sys}
 is
 \[
  -(\lambda_1+\lambda_3) \big( (\lambda_1-\lambda_2)^2 + (\lambda_2-\lambda_3)^2 \big),
 \]
 which is $0$ if and only if $\lambda_1+\lambda_3=0$ or $\lambda_1=\lambda_2=\lambda_3$.
If $\lambda_1+\lambda_3$ were 0, then $\lambda_2+\lambda_4$ would be $0$ as well, and, using that $\lambda_i\in\RR_+$, $i\in\{1,\ldots,4\}$,
 we would have that $\lambda_i=0$, $i\in\{1,\ldots,4\}$, leading us to a contradiction
 (since $\lambda_1+\lambda_2+\lambda_3+\lambda_4=1$).
If $\lambda_1=\lambda_2=\lambda_3$ were true, then, since $\lambda_4 = \lambda_1+\lambda_3 - \lambda_2$,
 we would have that $\lambda_4=\lambda_1$, and, using that $\lambda_1+\lambda_2+\lambda_3+\lambda_4=1$, it would imply that
 all the entries of $P$ would be equal to $\frac{1}{4}$, but we excluded this case.
All in all, the determinant of the (base) matrix of the system of linear equations \eqref{help_lin_sys} cannot be zero, and hence there exist a unique $\alpha,\beta,\gamma\in\RR$ such that \eqref{help_lin_sys} holds.
Consequently, the condition on $P$ in part (i) of Theorem \ref{Thm_Jensen_convexity_with_DS_weights}
 is satisfied by choosing $A:=\{1,2\}$ and $B:=\{3,4\}$, and hence we can apply part (i) of Theorem \ref{Thm_Jensen_convexity_with_DS_weights} in case of $\lambda_1+\lambda_3=\lambda_2+\lambda_4$
 as well.
\proofend

\section*{Declarations}

\subsection*{Funding}
The research of the second author was supported by the K-134191 NKFIH Grant.

\subsection*{Competing Interests} The authors have no relevant financial or non-financial interests to disclose.

\subsection*{Authors Contribution} Both authors contributed equally to the preparation of this manuscript.

\subsection*{Data availability} Not applicable.

\begin{thebibliography}{1}

\bibitem{BarPal_2025}
M.~Barczy and {\relax Zs}.~P{\'a}les.
\newblock A convexity-type functional inequality with infinite convex
  combinations.
\newblock {\em Ann. Univ. Sci. Budapest. Sect. Comput.}, 58:47--55, 2025.

\bibitem{BurMakSzo}
P.~Burai, J.~Mak\'{o}, and P.~Szokol.
\newblock Convexity generated by special circulant matrices.
\newblock {\em Bul. Ştiinţ. Univ. Politeh. Timiş. Ser. Mat. Fiz.},
  64:47--53, 2019.

\bibitem{CarEncGagJim}
A.~Carmona, A.~M. Encinas, S.~Gago, M.~J. Jim\'enez, and M.~Mitjana.
\newblock The inverses of some circulant matrices.
\newblock {\em Appl. Math. Comput.}, 270:785--793, 2015.

\bibitem{GilPal01}
A.~Gil\'anyi and {\relax Zs}.~P\'ales.
\newblock On {D}inghas-type derivatives and convex functions of higher order.
\newblock {\em Real Anal. Exchange}, 27(2):485--493, 2001/02.

\bibitem{Kuc2009}
M.~Kuczma.
\newblock {\em An Introduction to the Theory of Functional Equations and
  Inequalities}.
\newblock Birkh\"{a}user Verlag, Basel, 2nd edition, 2009.
\newblock Cauchy's equation and Jensen's inequality, Edited and with a preface
  by Attila Gil\'{a}nyi.

\bibitem{MarOlk}
A.~W. Marshall, I.~Olkin, and B.~C. Arnold.
\newblock {\em Inequalities: Theory of Majorization and its Applications}.
\newblock Springer Series in Statistics. Springer, New York, second edition,
  2011.

\bibitem{NicLar}
C.~P. Niculescu and L.-E. Persson.
\newblock {\em Convex Functions and their Applications---A Contemporary
  Approach}, volume~14 of {\em CMS/CAIMS Books in Mathematics}.
\newblock Springer, Cham, third edition, 2025.

\bibitem{Nie}
M.~Niezgoda.
\newblock Inequalities for convex functions and doubly stochastic matrices.
\newblock {\em Math. Inequal. Appl.}, 16(1):221--232, 2013.

\bibitem{NikPal03}
K.~Nikodem and {\relax Zs}.~P\'ales.
\newblock On {$t$}-convex functions.
\newblock {\em Real Anal. Exchange}, 29(1):219--228, 2003/04.

\end{thebibliography}

\end{document}